\documentclass{amsart}
\usepackage[english]{babel}
\usepackage{amsmath, amsthm, amssymb}
\usepackage{enumerate}
\usepackage{hyperref}
\usepackage{physics}
\usepackage[thinlines]{easytable}
\usepackage{algorithm}
\usepackage{algpseudocode}

\usepackage[a4paper, margin=1in]{geometry}

\usepackage{tikz}
\usetikzlibrary{arrows,shapes}
\usetikzlibrary{arrows.meta}
\usetikzlibrary{decorations.markings}
\usetikzlibrary{intersections}
\usetikzlibrary{positioning}

\theoremstyle{plain}
\newtheorem{theorem}{Theorem}

\newtheorem{lemma}[theorem]{Lemma}
\newtheorem{corollary}[theorem]{Corollary}
\theoremstyle{definition}
\newtheorem{example}[theorem]{Example}
\newtheorem*{remark}{Remark}

\newcommand{\R}{\mathbb{R}}
\newcommand{\Q}{\mathbb{Q}}
\newcommand{\Z}{\mathbb{Z}}

\newcommand{\D}{\mathcal{D}}
\renewcommand{\O}{\mathcal{O}}
\renewcommand{\P}{\mathcal{P}}
\newcommand{\lexle}{<_{\text{LEX}}}

\newcommand{\lexgeq}{\geq_{\text{LEX}}}

\DeclareMathOperator{\Nm}{N}
\DeclareMathOperator{\ord}{ord}

\makeatletter
\newcommand{\addresseshere}{
  \enddoc@text\let\enddoc@text\relax
}
\makeatother

\title{Partitions in real quadratic fields}
\author{David Stern}
\author{Mikul\'{a}\v{s} Zindulka}
\address{Charles University, Faculty of Mathematics and Physics, Department of Algebra,
Sokolovsk\'{a} 83, 186 75 Praha 8, Czech Republic}
\email{david.stern@matfyz.cuni.cz}
\email{mikulas.zindulka@matfyz.cuni.cz}

\thanks{We acknowledge support by Czech Science Foundation (GA\v{C}R) grant 21-00420M and the Charles University Research Centre program UNCE/SCI/022.}

\begin{document}

\begin{abstract}
We study partitions of totally positive integers in real quadratic fields. We develop an algorithm for computing the number of partitions, prove a result about the parity of the partition function, and characterize the quadratic fields such that there exists an element with exactly $1$--$5$, $7$, and $11$ partitions.
\end{abstract}

\maketitle

\section{Introduction}
The theory of integer partitions has been receiving constant attention since the times of Euler. In this paper we consider its analogy in a real quadratic field $ K $, where positive integers are replaced by totally positive integral elements. Virtually nothing is known about the properties of the associated partition function $ p_K(\alpha) $. Our aim is to open this new line of investigation.

The integer partition function $ p(n) $ is defined for $ n \in \Z_{\geq 1} $ as the number of integer partitions of $ n $. The famous asymptotic formula of Hardy and Ramanujan~\cite{HR} (subsequently improved by Rademacher~\cite{Rad1}) states that
\[
	p(n) \sim \frac{1}{4n\sqrt{3}}\exp(\pi\sqrt{\frac{2n}{3}}).
\]
(We use $ f(n) \sim g(n) $ to denote the fact that $ f(n)/g(n) \to 1 $ as $ n \to \infty $.)

Of course, there are many other properties of the partition function one might wish to investigate besides the asymptotics. Kolberg~\cite{Kol} was the first one to note that $ p(n) $ is odd as well as even infinitely often. His proof by contradiction is based on a recurrence for $ p(n) $ which is an immediate corollary of Euler's pentagonal number theorem. Parkin and Shanks~\cite{PS} conjectured that the natural density of $ n $ for which $ p(n) $ is odd equals $ 1/2 $.

A rich area of interest are congruence properties of $ p(n) $. Instead of trying to trace all the developments since the discovery of Ramanujan's congruences~\cite{Ram1, Ram2}, we refer to~\cite{BO,Ono}.

The notion of partition can be extended to the number field setting. Let $ K $ be a totally real number field with a ring of integers $ \O_K $ and let $ \O_K^+ $ denote the set of totally positive integers in $ K $. A \emph{partition} of an element $ \alpha \in \O_K^+ $ is an expression of the form
\[
	\alpha = \alpha_1+\alpha_2+\dots+\alpha_n,\qquad \alpha_i \in \O_K^+.
\]
The order of the summands is irrelevant. Let $ p_K(\alpha) $ be the number of partitions of $ \alpha $ and set $ p_K(0) = 1 $. We call $ p_K $ the \emph{partition function} associated with the number field $ K $.

The problem to estimate the growth of $ p_K $ was proposed by Rademacher~\cite{Rad2}. The solution was found by Meinardus, first in the case when $ K $ is real quadratic~\cite{Mei1} and then for an arbitrary totally real number field~\cite{Mei2}. If $ \Nm(\alpha) $ denotes the norm of $ \alpha $ in $ K $, then
\begin{equation}
\label{eqMei}
	\log p_K(\alpha) = (n+1)\cdot\sqrt[n+1]{\frac{\zeta(n+1)}{\sqrt{\Delta_K}}\Nm(\alpha)}\left(1+o(1)\right),\qquad \Nm(\alpha)\to\infty.
\end{equation}
Here $ \log $ denotes the natural logarithm, $ \zeta $ is the Riemann zeta function, $ n = [K:\Q] $ is the degree, and $ \Delta_K $ is the discriminant of $ K $. The result was further extended to an arbitrary number field (not necessarily totally real) by Mitsui~\cite{Mit}.

An element $ \alpha \in \O_K^+ $ is called \emph{indecomposable} if it cannot be expressed as $ \alpha = \beta+\gamma $ where $ \beta, \gamma \in \O_K^+ $. Thus indecomposable elements are precisely those $ \alpha \in \O_K^+ $ such that $ p_K(\alpha)=1 $. Dress and Scharlau~\cite{DS} characterized these elements in a real quadratic field $ K = \Q(\sqrt{D}) $ in terms of the continued fraction of $ \sqrt{D} $. Applications for indecomposables were found in the theory of universal quadratic forms over number fields~\cite{BK1, BK2, Kal, Kim, Yat}. Motivated by this, Hejda and Kala~\cite{HK} examined the additive structure of $ \O_K^+ $ and showed that it uniquely determines the underlying quadratic field.

One can also consider other types of partitions where the parts are elements of a real quadratic field $ K $ or an algebraic number field of a higher degree. One possibility is to fix $ \beta \in K $ and look at partitions of the form
\[
	\alpha = a_j\beta^j+a_{j-1}\beta^{j-1}+\dots+a_1\beta+a_0
\]
where $ j \in \Z_{\geq 0} $ and $ a_i \in \Z_{\geq 0} $ for $ i = 0, 1, \dots, j $. In the case when $ \beta = m \geq 2 $ is an integer, these are the so-called \emph{$ m $-ary partitions}~\cite{AFS, GMU, Mah, Zmi}. The generalization to an algebraic basis $ \beta $ was investigated by Kala and the second author~\cite{KZ}.

\

As far as we know, no one has studied the properties of the function $ p_K(\alpha) $ beyond the asymptotics. First of all, it is useful to be able to efficiently compute its values. To this end, we prove that $ p_K(\alpha) $ satisfies a recurrence similar to a well known formula for the integer partition function $ p(n) $.

\begin{theorem}
\label{thmRec}
Let $ K $ be a totally real number field. If $ \alpha \in \O_K^+ $, then
\begin{equation}
\label{eqRec}
	\alpha p_K(\alpha) = \sum_{0 \prec \beta \preceq \alpha}\sigma_K(\beta)p_K(\alpha-\beta),
\end{equation}
where the sum is taken over all $ \beta \in \O_K^+ $ such that $ \beta \preceq \alpha $ and $ \sigma_K $ is given by~\eqref{eqSigma}.
\end{theorem}

We find two applications for this theorem in real quadratic fields. First, we use it to develop an efficient algorithm for the computation of the values of $ p_K $ (Algorithm~\ref{algP}). The second application concerns the parity of the partition function.

\begin{theorem}
\label{thmParity}
Let $ K = \Q(\sqrt{D}) $ where $ D \in \Z_{\geq 2} $ is squarefree. If $ D \equiv 2, 3 \pmod{4} $, then there exist infinitely many $ n \in \Z_{\geq 1} $ such that $ p_K(n) $ is odd and infinitely many $ n \in \Z_{\geq 1} $ such that $ p_K(n) $ is even.
\end{theorem}

The method of proof of this theorem fails when $ D \equiv 1 \pmod{4} $. We do not know if the statement holds also in this case.

In the rest of the paper, we focus on the following problem: Let $ m \in \Z_{\geq 1} $ be given. Characterize all real quadratic fields $ K $ such that there exists an element $ \alpha\in\O_K^+ $ with exactly $ m $ partitions. To simplify the following discussion, we let
\[
	\D(m) = \left\{D \in \Z_{\geq 2}\text{ squarefree}|\; m \notin p_K(\O_K^+)\right\}.
\]
That is, $\D(m)$ is a set of $D$'s such that $ m $ does not belong to the range of $ p_K $ in $ K = \Q(\sqrt{D}) $.

We split the problem into two cases. In the first case, $ m $ belongs to the range of $ p(n) $, i.e., there exists $ n \in \Z_{\geq 1} $ with exactly $ m $ integer partitions. The first few values  of $ p(n) $ are $1$, $2$, $3$, $5$, $7$, and $11$. For these values, we can compute $ \D(m) $ explicitly.

\begin{theorem}
\label{thmD}
Let $ K = \Q(\sqrt{D}) $ where $ D \in \Z_{\geq 2} $ is squarefree. We have
\[
\begin{gathered}
	\D(1) = \emptyset,\quad\D(2) = \emptyset,\quad\D(3) = \{5\},\quad\D(5) = \{2, 3, 5\},\\
	\D(7) = \{2, 5\},\quad\D(11) = \{2, 3, 5, 6, 7, 13, 21\}.
\end{gathered}
\]
\end{theorem}

When $ m $ is not in the range of $ p(n) $, the situation is trickier. In this direction, we obtained results for $ m = 4 $ and $ 6 $. The quantity $ \xi_D $ which appears in the following theorem is defined in Section~\ref{secPrelim}.

\begin{theorem}
\label{thm46}
Let $ K = \Q(\sqrt{D}) $ where $ D \in \Z_{\geq 2} $ is squarefree.
\begin{enumerate}[i)]
\item There exists $ \alpha \in \O_K^+ $ such that $ p_K(\alpha) = 4 $.
\item If $ \lceil \xi_D \rceil-\xi_D > \frac{1}{2} $, then there exists $ \alpha \in \O_K^+ $ such that $ p_K(\alpha) = 6 $.
\end{enumerate}
\end{theorem}

Moreover, an element $ \alpha $ with exactly $ 4 $ partitions is explicitly constructed in Theorem~\ref{thmM4}. The condition for $ 6 $ to be contained in the range of $ p_K $ is sufficient but not necessary. The first counterexample is $ K = \Q(\sqrt{3}) $ (but see also Example~\ref{example6}). Thus if $ D = 3 $, then $ \lceil\xi_D\rceil-\xi_D = \lceil\sqrt{3}\rceil-\sqrt{3} \leq \frac{1}{2} $ but $ \alpha = 4 $ has $ 6 $ partitions:
\[
	4 = 3+1 = 2+2 = 2+1+1 = 1+1+1+1 = \left(2+\sqrt{3}\right)+\left(2-\sqrt{3}\right).
\]

The results of our numerical computations are included in the Appendix. Tables~\ref{tabD2}--\ref{tabD21} contain the number of partitions of particular elements in various quadratic fields. Tables~\ref{tabD23mod4} and~\ref{tabD1mod4} contain all elements (up to conjugation and multiplication by units) with a given number of partitions.

\section*{Acknowledgements}
This article is partly based on a bachelor thesis of the first author~\cite{Ste}. We wish to thank our supervisor V\'{i}t\v{e}zslav Kala for his guidance and support and Daniel Gil-Mu\~{n}oz for helpful comments.

\section{Preliminaries}
\label{secPrelim}
Let $ K $ be a totally real number field of degree $ d = [K:\Q] $, so that there exist $ d $ real embeddings
\[
	\sigma_1, \sigma_2, \ldots, \sigma_d: K \hookrightarrow \R.
\]
An element $ \alpha \in K $ is called \emph{totally positive} if
\[
	\sigma_i(\alpha)>0, \qquad i = 1, \ldots, d.
\]
If $ \alpha, \beta \in K $, we write $ \alpha \succ \beta $ to denote the fact that $ \alpha-\beta $ is totally positive. In particular, $ \alpha \succ 0 $ means that $ \alpha $ is totally positive. We use $ \alpha \succeq \beta $ if $ \alpha \succ \beta $ or $ \alpha = \beta $.

We mostly focus on the case when $ K = \Q(\sqrt{D}) $ is a real quadratic field. We always assume that $ D \geq 2 $ is a squarefree integer. In a real quadratic field, $ \alpha $ is totally positive if and only if $ \alpha > 0 $ and $ \alpha' > 0 $, that is, if and only if both the element and its conjugate are positive.

Now we describe the characterization of indecomposables originally proved in~\cite{DS}. These facts are fairly standard and their summary can also be found in~\cite{HK}. Throughout the paper, we use the following notation. The ring of integers $ \O_K $ has a basis $ (1, \omega_D) $ where
\[
	\omega_D := \begin{cases}
		\sqrt{D},&D \equiv 2, 3 \pmod{4},\\
		\frac{1+\sqrt{D}}{2},&D \equiv 1 \pmod{4}.
	\end{cases}
\]
We define
\[
	\xi_D := -\omega_D' = \begin{cases}
		\sqrt{D},&D \equiv 2, 3 \pmod{4},\\
		\frac{\sqrt{D}-1}{2},&D \equiv 1 \pmod{4}
	\end{cases}
\]
and
\[
	\sigma_D := \omega_D+\lfloor\xi_D\rfloor = \begin{cases}
		\sqrt{D}+\lfloor\sqrt{D}\rfloor,& D \equiv 2, 3 \pmod{4},\\
		\frac{1+\sqrt{D}}{2}+\lfloor\frac{\sqrt{D}-1}{2}\rfloor,& D \equiv 1 \pmod{4}.
	\end{cases}
\]

The element $ \sigma_D $ has a purely periodic continued fraction expansion
\[
	\sigma_D = [\overline{u_0, u_1, \ldots, u_{s-1}}].
\]
Let $ p_{-1}=1 $, $ q_{-1}=0 $ and $ p_0 = \lceil u_0/2 \rceil $, $ q_0 = 1 $. Let $ p_i $ and $ q_i $ be defined recursively by
\begin{alignat*}{2}
	p_{i+2}& := u_{i+2}p_{i+1}+p_i,\quad & i\geq -1,\\
	q_{i+2}& := u_{i+2}q_{i+1}+q_i,\quad & i\geq -1.
\end{alignat*}
Then $ p_i/q_i = [\lceil u_0/2, u_1, \dots, u_i] $ are the convergents to $ \omega_D $. Next, we define
\begin{alignat*}{2}
	\alpha_i& := p_i+q_i\xi_D,\quad & i \geq -1,\\
	\alpha_{i, r}& := \alpha_i+r\alpha_{i+1},\quad & i \geq -1.
\end{alignat*}

The indecomposable elements in $ \O_K^+ $ are characterized as follows: they are the elements $ \alpha_{i, r} $ with odd $ i \geq -1 $ and $ 0 \leq r \leq u_{i+2}-1 $ together with their conjugates.

If we let $ \varepsilon>1 $ denote the fundamental unit of $ \O_K $, then $ \varepsilon = \alpha_{s-1} $. The smallest totally positive unit $ \varepsilon_+>1 $ in $ \O_K^+ $ satisfies $ \varepsilon_+ = \varepsilon = \alpha_{s-1} $ if $ s $ is even and $ \varepsilon_+ = \varepsilon^2 = \alpha_{2s-1} $ if $ s $ is odd. In general, the $ \alpha_i $'s satisfy $ \alpha_{i+s} = \varepsilon \alpha_i $, hence also $ \alpha_{i+s, r} = \varepsilon \alpha_{i, r} $. Thus there exist only finitely many indecomposables up to multiplication by units.

\section{Recurrence for the partition function}
\label{secRec}

In order to prove the estimate~\eqref{eqMei} for $ p_K(\alpha) $ in real quadratic fields, one uses the generating function in two variables
\[
	1+\sum_{\alpha\succ 0}p_K(\alpha)x^\alpha y^{\alpha'} = \prod_{\beta\succ 0} \frac{1}{1-x^\beta y^{\beta'}},\qquad |x|<1, |y|<1.
\]
The sum and the product are taken over elements in $ \O_K^+ $. For a proof of the equality of the two expressions in the stated range, see~\cite[Hilfssatz 1]{Mei1}.
For our purposes, it is more convenient to use a different generating function. By an exact analogy with the integer partitions, we set
\[
	F_K(x) = \sum_{\alpha\succeq 0}p_K(\alpha)x^\alpha = \prod_{\beta\succ 0}\frac{1}{1-x^\beta},
\]
where the sum and the product are extended over all $ \alpha \in \O_K^+ $. However, it should be noted straight away that the sum and the product converge only at the point $ x = 0 $. Indeed, there exists a sequence $ (\alpha_j)_{j=1}^\infty \subset \O_K^+ $ such that $ \alpha_j \to 0 $. Thus the series and the product in the definition of $ F_K(x) $ can be treated only as formal expressions.

The integer partition function $ p(n) $ satisfies the recurrence
\[
	np(n) = \sum_{k=1}^n\sigma(k)p(n-k),\qquad n \geq 1
\]
where $ \sigma(k) $ is the sum of divisors of $ k $ (see for example~\cite[p.~9]{Wil}). To generalize this formula to totally real number fields, we first define an analogue of the divisor function $ \sigma $.

Let $ K $ be a number field. If $ \beta \in \O_K $, let
\begin{equation}
\label{eqSigma}
	\sigma_K(\beta) := \sum_{n \in \Z_{\geq 1}, n \mid \beta}\frac{\beta}{n}.
\end{equation}

If we let $ c(\beta) $ denote the largest $ n \in \Z_{\geq 1} $ such that $ n \mid \beta $, then
\[
	\sigma_K(\beta) = \beta \sum_{n \in \Z_{\geq 1}, n \mid c(\beta)}\frac{1}{n} = \frac{\beta}{c(\beta)}\sigma(c(\beta)).
\]

\begin{proof}[Proof of Theorem~\ref{thmRec}]
The recurrence is proved by the standard technique of taking the logarithmic derivative. Applying the logarithmic derivative to the generating function $ F_K $, we get
\begin{align*}
	x\frac{F_K'(x)}{F_K(x)}& = x\dv{x}\sum_{\alpha \succ 0}-\log(1-x^\alpha) = x\dv{x}\sum_{\alpha \succ 0}\sum_{k=1}^\infty \frac{x^{k\alpha}}{k} = \sum_{\alpha \succ 0}\sum_{k=1}^\infty \alpha x^{k\alpha}\\
	& = \sum_{\beta \succ 0}\left(\sum_{k \mid \beta}\frac{\beta}{k}\right)x^\beta = \sum_{\beta \succ 0}\sigma_K(\beta)x^\beta,
\end{align*}
hence
\[
	\sum_{\alpha \succ 0}\alpha p_K(\alpha)x^\alpha = \left(\sum_{\beta \succ 0}\sigma_K(\beta)x^\beta\right)\left(\sum_{\gamma \succeq 0}p_K(\gamma)x^\gamma\right).
\]
The equality follows by comparing the coefficients of $ x^\alpha $ on both sides.
\end{proof}
We describe next how to use Theorem~\ref{thmRec} to compute $ p_K(\alpha) $ when $ K $ is a real quadratic field. The lexicographical ordering on $ K $ is defined as follows: If $ \alpha $ and $ \beta $ are two elements in $ K $, let $ \alpha = a+b\omega_D $ and $ \beta = c+d\omega_D $ be their expressions in the integral basis $ (1, \omega_D) $. We write $ \alpha \lexle \beta $ if either $ a < c $, or $ a = c $ and $ b < d $. There is an inclusion between the relations $ \prec $ and $ \lexle $.
\begin{lemma}
\label{lemmaOrd}
Let $ K = \Q(\sqrt{D}) $ where $ D \in \Z_{\geq 2} $ is squarefree. If $ \alpha, \beta \in K $, then
\[
	\alpha \prec \beta \Longrightarrow \alpha \lexle \beta.
\]
\end{lemma}
\begin{proof}
Let $ \alpha = a+b\omega_D $ and $ \beta = c+d\omega_D $. Suppose that $ \alpha \lexle \beta $ is not satisfied, so that either $ a > c $, or $ a = c $ and $ b \geq d $.

Assume first that $ a > c $. If $ b \geq d $, then $ \alpha = a+b\omega_D > c+d\omega_D = \beta $. If $ b < d $, then $ b\omega_D' > d\omega_D' $, and hence $ \alpha' = a+b\omega_D' > c+d\omega_D' = \beta' $. In either case, $ \alpha \prec \beta $ is not satisfied.

Assume next that $ a = c $ and $ b \geq d $. We get $ \alpha = a+b\omega_D \geq c+d\omega_D = \beta $, so that $ \alpha \prec \beta $ is again not satisfied.
\end{proof}

Let $ \alpha = x+y\omega_D $ where $ x \in \Z_{\geq 1} $ and $ y \in \Z $. The element $ \alpha $ is totally positive if and only if $ \alpha = x+y\omega_D > 0 $ and $ \alpha' = x+y\omega_D' > 0 $, which holds if and only if $ \lceil -x/\omega_D \rceil \leq y \leq \lfloor x/\xi_D \rfloor $.

When we compute the values of $ p_K $ for different $ \alpha \in \O_K^+ $ using Theorem~\ref{thmRec}, we make sure to run through the successive $ \alpha $'s in the lexicographical ordering. This way, the values $ p_K(\alpha-\beta) $ appearing in the sum are already computed by the time we compute $ p_K(\alpha) $. The algorithm written in pseudocode is given below.

\begin{algorithm}
\begin{algorithmic}
\State\textbf{Input:} A real quadratic field $ K $ with the integral basis $ (1, \omega_D) $ and a positive integer $ M$.
\State\textbf{Output:} A table $ T $ of the values $ p_K(\alpha) $ for $ \alpha \in \O_K^+ $ of the form $ \alpha = x+y\omega_D $, where $ x = 1, \dots, M $.
\State Set $ p_K(0) = 1 $.
\For{$ x = 1, \ldots, M $}
	\For{$ y = \lceil -x/\omega_D \rceil, \ldots, \lfloor x/\xi_D \rfloor $}
		\State Compute $ p_K(x+y\omega_D) $ from~\eqref{eqRec} and set $ T[x, y] = p_K(x+y\omega_D) $.
	\EndFor
\EndFor
\State Output $ T $.
\end{algorithmic}
\caption{Computation of $ p_K(x+y\omega_D) $}
\label{algP}
\end{algorithm}

As an example, we computed the number of partitions in the field $ K = \Q(\sqrt{2}) $ (Figure~\ref{figPK}).

\begin{figure}[ht]
    \centering
    \begin{tikzpicture}[scale=0.8]
    \draw[->] (0, 0)--(11.5, 0);
    \node[right] at (11.5, 0) {$ x $};
    \draw[->] (0, 0)--(0, 8.5);
    \node[left] at (0, 8.5) {$ y $};
    \node[below] at (0.5, 0) {$ \mathbf{0} $};
    \node[below] at (1.5, 0) {$ \mathbf{1} $};
    \node[below] at (2.5, 0) {$ \mathbf{2} $};
    \node[below] at (3.5, 0) {$ \mathbf{3} $};
    \node[below] at (4.5, 0) {$ \mathbf{4} $};
    \node[below] at (5.5, 0) {$ \mathbf{5} $};
    \node[below] at (6.5, 0) {$ \mathbf{6} $};
    \node[below] at (7.5, 0) {$ \mathbf{7} $};
    \node[below] at (8.5, 0) {$ \mathbf{8} $};
    \node[below] at (9.5, 0) {$ \mathbf{9} $};
    \node[below] at (10.5, 0) {$ \mathbf{10} $};
    \node[left] at (0, 0.5) {$ \mathbf{0} $};
    \node[left] at (0, 1.5) {$ \mathbf{1} $};
    \node[left] at (0, 2.5) {$ \mathbf{2} $};
    \node[left] at (0, 3.5) {$ \mathbf{3} $};
    \node[left] at (0, 4.5) {$ \mathbf{4} $};
    \node[left] at (0, 5.5) {$ \mathbf{5} $};
    \node[left] at (0, 6.5) {$ \mathbf{6} $};
    \node[left] at (0, 7.5) {$ \mathbf{7} $};
    \draw[-] (0, 1)--(11, 1);
    \draw[-] (0, 2)--(11, 2);
    \draw[-] (0, 3)--(11, 3);
    \draw[-] (0, 4)--(11, 4);
    \draw[-] (0, 5)--(11, 5);
    \draw[-] (0, 6)--(11, 6);
    \draw[-] (0, 7)--(11, 7);
    \draw[-] (0, 8)--(11, 8);
    \draw[-] (1, 0)--(1, 8);
    \draw[-] (2, 0)--(2, 8);
    \draw[-] (3, 0)--(3, 8);
    \draw[-] (4, 0)--(4, 8);
    \draw[-] (5, 0)--(5, 8);
    \draw[-] (6, 0)--(6, 8);
    \draw[-] (7, 0)--(7, 8);
    \draw[-] (8, 0)--(8, 8);
    \draw[-] (9, 0)--(9, 8);
    \draw[-] (10, 0)--(10, 8);
    \draw[-] (11, 0)--(11, 8);
    \node at (0.5, 0.5) {1};
    \node at (1.5, 0.5) {1};
    \node at (2.5, 0.5) {2};
    \node at (3.5, 0.5) {3};
    \node at (4.5, 0.5) {6};
    \node at (5.5, 0.5) {10};
    \node at (6.5, 0.5) {19};
    \node at (7.5, 0.5) {34};
    \node at (8.5, 0.5) {62};
    \node at (9.5, 0.5) {108};
    \node at (10.5, 0.5) {190};
    \node at (0.5, 1.5) {0};
    \node at (1.5, 1.5) {0};
    \node at (2.5, 1.5) {1};
    \node at (3.5, 1.5) {2};
    \node at (4.5, 1.5) {4};
    \node at (5.5, 1.5) {8};
    \node at (6.5, 1.5) {16};
    \node at (7.5, 1.5) {29};
    \node at (8.5, 1.5) {54};
    \node at (9.5, 1.5) {98};
    \node at (10.5, 1.5) {175};
    \node at (0.5, 2.5) {0};
    \node at (1.5, 2.5) {0};
    \node at (2.5, 2.5) {0};
    \node at (3.5, 2.5) {1};
    \node at (4.5, 2.5) {3};
    \node at (5.5, 2.5) {6};
    \node at (6.5, 2.5) {12};
    \node at (7.5, 2.5) {23};
    \node at (8.5, 2.5) {44};
    \node at (9.5, 2.5) {81};
    \node at (10.5, 2.5) {149};
    \node at (0.5, 3.5) {0};
    \node at (1.5, 3.5) {0};
    \node at (2.5, 3.5) {0};
    \node at (3.5, 3.5) {0};
    \node at (4.5, 3.5) {0};
    \node at (5.5, 3.5) {2};
    \node at (6.5, 3.5) {6};
    \node at (7.5, 3.5) {13};
    \node at (8.5, 3.5) {28};
    \node at (9.5, 3.5) {56};
    \node at (10.5, 3.5) {107};
    \node at (0.5, 4.5) {0};
    \node at (1.5, 4.5) {0};
    \node at (2.5, 4.5) {0};
    \node at (3.5, 4.5) {0};
    \node at (4.5, 4.5) {0};
    \node at (5.5, 4.5) {0};
    \node at (6.5, 4.5) {2};
    \node at (7.5, 4.5) {6};
    \node at (8.5, 4.5) {16};
    \node at (9.5, 4.5) {33};
    \node at (10.5, 4.5) {69};
    \node at (0.5, 5.5) {0};
    \node at (1.5, 5.5) {0};
    \node at (2.5, 5.5) {0};
    \node at (3.5, 5.5) {0};
    \node at (4.5, 5.5) {0};
    \node at (5.5, 5.5) {0};
    \node at (6.5, 5.5) {0};
    \node at (7.5, 5.5) {0};
    \node at (8.5, 5.5) {4};
    \node at (9.5, 5.5) {13};
    \node at (10.5, 5.5) {33};
    \node at (0.5, 6.5) {0};
    \node at (1.5, 6.5) {0};
    \node at (2.5, 6.5) {0};
    \node at (3.5, 6.5) {0};
    \node at (4.5, 6.5) {0};
    \node at (5.5, 6.5) {0};
    \node at (6.5, 6.5) {0};
    \node at (7.5, 6.5) {0};
    \node at (8.5, 6.5) {0};
    \node at (9.5, 6.5) {3};
    \node at (10.5, 6.5) {12};
    \node at (0.5, 7.5) {0};
    \node at (1.5, 7.5) {0};
    \node at (2.5, 7.5) {0};
    \node at (3.5, 7.5) {0};
    \node at (4.5, 7.5) {0};
    \node at (5.5, 7.5) {0};
    \node at (6.5, 7.5) {0};
    \node at (7.5, 7.5) {0};
    \node at (8.5, 7.5) {0};
    \node at (9.5, 7.5) {0};
    \node at (10.5, 7.5) {1};
    \end{tikzpicture}
	\caption{\label{figPK}The values $ p_K(x+y\sqrt{2}) $ of the partition function in $ K = \Q(\sqrt{2}) $.}
\end{figure}

\section{Parity of the partition function}
\label{secParity}

In this section, we show that if $ K = \Q(\sqrt{D}) $, where $ D \equiv 2, 3 \pmod{4} $, then there exist infinitely many $ n \in \Z_{\geq 1} $ such that $ p_K(n) $ has a prescribed parity. For this purpose we define a ``cumulative" partition function $ P_K $ on $ \Z_{\geq 1} $ by
\[
	P_K(n) = \sum_{\alpha\succ 0, \Tr\alpha = 2n}p_K(\alpha).
\]
The sum runs over all $ \alpha \in \O_K^+ $ with trace equal to $ 2n $. We set $ P_K(0) = 1 $.
\begin{lemma}
\label{lemmaParity}
Let $ K = \Q(\sqrt{D}) $ where $ D \in \Z_{\geq 2} $ is squarefree. If $ n \in \Z_{\geq 0} $, then
\[
	P_K(n) \equiv p_K(n)\pmod{2}.
\]
\end{lemma}
\begin{proof}
The totally positive elements $ \alpha = a+b\omega_D $ with $ b\neq 0 $ can be paired with their conjugates $ \alpha' $. Since $ p_K(\alpha) = p_K(\alpha') $, we obtain
\begin{align*}
	P_K(n)& = p_K(n)+\sum_{\substack{\alpha\succ 0, \Tr\alpha = 2n\\\alpha \neq n}}p_K(\alpha)+p_K(\alpha')\\
	& = p_K(n)+2\sum_{\substack{\alpha\succ 0, \Tr\alpha = 2n\\\alpha \neq n}}p_K(\alpha)\\
	& \equiv p_K(n) \pmod{2}.
\end{align*}
\end{proof}
Next we discuss the theory of prefabs (as described in~\cite[Section~4]{Wil}) and how it relates to partitions of totally positive integers. Let $ \P $ be a collection of objects, called \emph{prefabs}. We assume that there is an operation $ \oplus $ on $ \P $, so that for two prefabs $ P_1 $ and $ P_2 $, we can take their composition $ P_1\oplus P_2 $. We define a function $ \ord: \P \to \Z_{\geq 0} $, which assigns an order to each prefab. The function is assumed to be additive, so that $ \ord(P_1\oplus P_2) = \ord(P_1)+\ord(P_2) $. There are also certain objects in $ \P $, called \emph{primes}, such that each element of $ \P $ can be written uniquely as a composition of primes.

From now on, we work under the assumption that $ K = \Q(\sqrt{D}) $, where $ D \equiv 2, 3 \pmod{4} $. In this setting, the prefabs are partitions of totally positive integers and the primes are totally positive integers themselves. The order function is defined on the totally positive integers by $ \alpha \mapsto \frac{\Tr\alpha}{2} $ and is then extended additively.

Let $ a_n $ be the number of $ \alpha \in \O_K^+ $ of order $ n $, that is,
\[
	a_n = \#\left\{\alpha \in \O_K^+|\; \Tr\alpha = 2n\right\}.
\]
The function $ P_K(n) $ then gives the number of prefabs (i.e. partitions of totally positive integers) of order $ n $. We know from the general theory of prefabs (see~\cite[Lemma~3]{Wil}) that $ P_K(n) $ satisfies the product formula
\[
	\sum_{n=0}^\infty P_K(n)x^n = \prod_{j = 1}^\infty \frac{1}{(1-x^j)^{a_j}}.
\]
This is easy to see intuitively: we can choose at most $ a_j $ different prime objects of order $ j $ in the product.

By taking the logarithmic derivative, a generating function of this form can be quickly shown to satisfy the recurrence
\[
	nP_K(n) = \sum_{k=1}^n\left(\sum_{d \mid k}d a_d\right)P_K(n-k),\qquad n \in \Z_{\geq 1}.
\]
Now we are ready to prove Theorem~\ref{thmParity}.
\begin{proof}[Proof of Theorem~\ref{thmParity}]
First we prove the statement for the function $ P_K $ instead of $ p_K $. Then we appeal to Lemma~\ref{lemmaParity}.

A crucial observation is that the $ a_d $ are all odd. Thus
\[
	nP_K(n) = \sum_{k=1}^n\left(\sum_{d\mid k}da_d\right)P_K(n-k) \equiv \sum_{k=1}^n\left(\sum_{d\mid k}d\right)P_K(n-k) = \sum_{k=1}^n \sigma(k)P_K(n-k)\pmod{2}.
\]
Now we proceed analogously as we would in the case of the integer partition function. Observe that $ \sigma(k) $ is odd if and only if $ k = m^2 $ or $ k = 2m^2 $ for some $ m \in \Z_{\geq 1} $. If we let
\[
	S(n) := \{m^2|\;1 \leq m^2 \leq n\}\cup\{2m^2|\;2 \leq 2m^2 \leq n\},
\]
then
\[
	nP_K(n) \equiv \sum_{k \in S(n)}P_K(n-k) \pmod{2}.
\]
In order to prove that there exist infinitely many $ n $ such that $ P_K(n) $ is odd, suppose for contradiction that $ P_K(n) $ is even for all $ n \geq n_0 $. Choose $ n = 2m^2+1 $ such that for every $ k \in S(n)\setminus\{2m^2\} $, $ n-k \geq n_0 $. Then
\[
	(2m^2+1)P_K(n) \equiv \sum_{k\in S(n)\setminus\{2m^2\}}P_K(n-k)+P_K(1) \equiv 1 \pmod{2},
\]
hence $ P_K(n) $ is odd, a contradiction. The proof that $ P_K(n) $ is even for infinitely many $ n $ is similar.
\end{proof}

\section{Values of the partition function}
\label{secVal1}

The general question we investigate in the present section is: Given a positive integer $ m $, for which real quadratic fields $ K $ does there exist $ \alpha \in \O_K^+ $ such that $ p_K(\alpha)=m $?

For $ m \in \{1, 2\} $, the answer is immediate because $ p_K(1) = 1 $ and $ p_K(2) = 2 $ for every real quadratic field $ K $. In fact, every unit $ \varepsilon \in \O_K^+ $ is indecomposable, i.e., $ p_K(\varepsilon) = 1 $, from which $ p_K(1)=1 $ follows as a special case. If $ 2 = \alpha_1+\alpha_2+\ldots+\alpha_j $ is an arbitrary partition of $ 2 $, each part $ \alpha_i $ satisfies $ \alpha_i \preceq 2 $. Thus the only possibilities are $ \alpha_i = 1 $ or $ 2 $, leading to the two partitions $ 2 = 1+1 $. 

\begin{theorem}
\label{thmDn}
Let $ K = \Q(\sqrt{D}) $ where $ D \in \Z_{\geq 2} $ is squarefree and let $ n \in \Z_{\geq 1} $. There exists $ D_n > 0 $ such that if $ D > D_n $, then
\[
	p_K(n) = p(n).
\]
Moreover, let $ E_n := \left\lfloor\frac{n}{2}\right\rfloor^2 $ and
\[
	F_n := \begin{cases}
		(n-1)^2,&n\text{ even},\\
		n^2,&n\text{ odd}.
	\end{cases}
\]
\begin{itemize}
\item If $ D \equiv 2, 3 \pmod{4} $ and $ D> E_n $, then $ p_K(n) = p(n) $.
\item If $ D \equiv 1 \pmod{4} $ and $ D > F_n $, then $ p_K(n) = p(n) $.
\end{itemize}
The bounds $ E_n $ and $ F_n $ are optimal.
\end{theorem}
\begin{proof}
We first prove the existence of $ D_n $. Let
\[
	n = \alpha_1+\alpha_2+\dots+\alpha_j,\qquad \alpha_i \in \O_K^+,
\]
be a non-trivial partition of $ n $, where $ j \geq 2 $ and $ \alpha_i = a_i+b_i\omega_D $ for $ 1 \leq i \leq j $. Since $ \alpha_i \preceq n $, we get $ a_i \leq n $ for $ 1 \leq i \leq j $. We want to show that for $ D $ large enough, all the $ b_i $'s must be equal to zero.

The inequalities $ a_i+b_i\omega_D > 0 $ and $ a_i+b_i\omega_D' > 0 $ imply
\[
	-\frac{n}{\omega_D} \leq -\frac{a_i}{\omega_D} < b_i < \frac{a_i}{\xi_D} \leq \frac{n}{\xi_D}.
\]
Because $ \omega_D \to +\infty $ and $ \xi_D \to +\infty $ as $ D \to \infty $, there exists $ D_n $ such that if $ D > D_n $, then $ b_i = 0 $ for $ 1 \leq i \leq j $.

Without loss of generality, we may assume that the parts $ \alpha_i $ are arranged lexicographically:
\[
	\alpha_1 \lexgeq \alpha_2 \lexgeq \dots \lexgeq \alpha_j,
\]
and in particular $ a_1 \geq a_2 \geq \dots \geq a_j $. Now $ n = (a_1+a_2+\dots+a_j)+(b_1+b_2+\dots+b_j)\omega_D $ implies $ \sum_{i=1}^j a_i = n $ and $ \sum_{i=1}^j b_i = 0 $. Write $ n $ as $ n = 2k $ or $ n = 2k+1 $ depending on its parity. We have $ 2a_2 \leq a_1+a_2 \leq n $, hence $ a_2 \leq k $.

\

Assume $ D \equiv 2, 3 \pmod{4} $. We will show that if $ D > E_n = k^2 $, then all the $ b_i $'s must be equal to zero. Since $ a_2 \leq k $, we also have $ a_i \leq k $ for $ 2 \leq i \leq j $, and
\[
	-1 < -\frac{k}{\sqrt{D}} \leq -\frac{a_i}{\omega_D} < b_i < \frac{a_i}{\xi_D} \leq \frac{k}{\sqrt{D}} < 1.
\]
Thus $ b_i = 0 $ for $ 2 \leq i \leq j $ and since $ \sum_{i=1}^j b_i = 0 $, we also get $ b_1 = 0 $.

Now we show that the bound $ E_n $ is optimal, i.e., if $ D \leq E_n $, then $ p_K(n) > p(n) $. For $ n \in \{1, 2, 3\} $, the bounds are $ E_1 = 0 $, $ E_2 = 1 $, $ E_3 = 1 $. Let $ n \geq 4 $ and $ D \leq E_n $. We can assume $ D < E_n = k^2 $ because $ D $ is squarefree. If $ n = 2k $, then
\[
	n = \left(k+\sqrt{D}\right)+\left(k-\sqrt{D}\right)
\]
is a partition of $ n $ and if $ n = 2k+1 $, then
\[
	n = \left(k+1+\sqrt{D}\right)+\left(k-\sqrt{D}\right)
\]
is a partition of $ n $. Thus $ p_K(n) > p(n) $ in both cases.

\

Next, we assume $ D \equiv 1 \pmod{4} $ and we want to show that if $ D > F_n $, then all the $ b_i $'s must be equal to zero. In the case $ n = 2k $, the bound is $ F_n = (2k-1)^2 $. We get
\[
	\omega_D = \frac{1+\sqrt{D}}{2} > k,\qquad \xi_D = \frac{\sqrt{D}-1}{2} > k-1.
\]
If $ a_2 = k $, then $ a_1 \geq a_2 $ implies $ a_1 = k $ and the partition is of the form
\[
	2k = \left(k+b_1\omega_D\right)+\left(k+b_2\omega_D\right),
\]
from which it follows that $ b_2 = -b_1 $. Since $ \alpha_1 \lexgeq \alpha_2 $, we have $ b_1 \geq 0 $ and $ b_2 \leq 0 $. We get
\[
	-1< -\frac{k}{\omega_D} = -\frac{a_2}{\omega_D} < b_2 \leq 0,
\]
hence $ b_2 = 0 $ and $ b_1 = 0 $.

If $ a_2 \leq k-1 $, then $ a_i \leq k-1 $ for $ 2 \leq i \leq j $, and
\[
	-1 < -\frac{k-1}{\omega_D} < -\frac{a_i}{\omega_D} < b_i < \frac{a_i}{\xi_D} \leq \frac{k-1}{\xi_D} < 1,
\]
hence $ b_i = 0 $ for $ 2 \leq i \leq j $, and consequently $ b_1 = 0 $.

In the case $ n = 2k+1 $, the bound is $ F_n = (2k+1)^2 $, and we can estimate
\[
	\omega_D = \frac{1+\sqrt{D}}{2} > k+1,\qquad \xi_D = \frac{\sqrt{D}-1}{2} > k.
\]
Thus we get
\[
	-1 < -\frac{k}{\omega_D} \leq -\frac{a_i}{\omega_D} < b_i < \frac{a_i}{\xi_D} \leq \frac{k}{\xi_D} < 1,
\]
for $ 2 \leq i \leq j $, and again $ b_i = 0 $ for all $ 1 \leq i \leq j $.

Finally, let us show that the bound $ F_n $ is optimal, i.e., if $ D \leq F_n $, then $ p_K(n) > p(n) $. For $ n \in \{1, 2\} $, the bounds are $ F_1 = F_2 = 1 $. Let $ n \geq 3 $ and $ D < F_n $. If $ n = 2k $ is even, then $ F_n = (2k-1)^2 $ and
\[
	n = \left(k+\omega_D\right)+\left(k-\omega_D\right)
\]
is a partition. If $ n = 2k+1 $ is odd, then $ F_n = (2k+1)^2 $ and
\[
	n = \left(k+1-\omega_D\right)+\left(k+\omega_D\right)
\]
is a partition. Thus $ p_K(n) > p(n) $ in both cases.
\end{proof}

Using the bounds $ E_n $ and $ F_n $ from the preceding theorem, we get the following corollary.

\begin{corollary}
\label{corP}
Let $ K = \Q(\sqrt{D}) $ where $ D \in \Z_{\geq 2} $ is squarefree.
\begin{itemize}
\item $ p_K(3) = 3 $ for $ D \in \{2, 3\} $ and $ D > 5 $,
\item $ p_K(4) = 5 $ for $ D > 5 $,
\item $ p_K(5) = 7 $ for $ D \equiv 2, 3 \pmod{4} $, $ D > 4 $, and for $ D \equiv 1 \pmod{4} $, $ D > 25 $,
\item $ p_K(6) = 11 $ for $ D \equiv 2, 3 \pmod{4} $, $ D > 9 $, and for $ D \equiv 1 \pmod{4} $, $ D > 25 $.
\end{itemize}
\end{corollary}

To be able to characterize the real quadratic fields $ K $ where $ m $ is in the range of $ p_K $ for $ m \in\{3, 5, 7, 11\} $, we turn our attention to the following question: How to decide whether there exists $ \alpha \in \O_K^+ $ such that $ p_K(\alpha)=m $ when $ K $ is fixed?

Recall that $ \varepsilon_+>1 $ denotes the smallest totally positive unit. If $ \alpha \in \O_K^+ $, then there exists an associated integral element $ \beta = \eta\alpha $ where $ \eta \in \O_K^+ $ is a totally positive unit such that
\[
	\varepsilon_+^{-1} < \frac{\beta}{\beta'} \leq \varepsilon_+.
\]
Given that $ p_K(\alpha) = p_K(\beta) $, we may restrict our attention to the integral elements satisfying these inequalities. If we express $ \alpha $ in the integral basis as $ \alpha = x+y\omega_D $, then $ x \geq 1 $ because $ \alpha $ is totally positive. Moreover, either $ \alpha $ or its conjugate $ \alpha' $ have the second coordinate non-negative, so we may also assume $ y\geq 0 $. An element $ \alpha $ with $ x \geq 1 $ and $ y \geq 0 $ satisfies $ 1 \leq \alpha/\alpha' $.

To summarize: In the search for $ \alpha \in \O_K^+ $ satisfying $ p_K(\alpha)=m $, we may consider only $ \alpha = x+y\omega_D $ with $ x \geq 1 $, $ y \geq 0 $, and such that
\[
	1 \leq \frac{\alpha}{\alpha'} \leq \varepsilon_+.
\]
If $ \alpha = x+y\omega_D \succ 0 $, then $ x > y\xi_D $, hence $ \alpha $ can be expressed as
\[
	\alpha = (\lceil y\xi_D \rceil+k)+y\omega_D,\qquad k \in \Z_{\geq 0}.
\]

We prove an easy lemma which helps with the computations.

\begin{lemma}
\label{lemmaK}
Let $ K = \Q(\sqrt{D}) $ where $ D \in \Z_{\geq 2} $ is squarefree.
\begin{enumerate}[i)]
\item For every $ m \in \Z_{\geq 1} $, there exists $ y \in \Z_{\geq 0} $ such that $ p_K\left(\lceil y\xi_D \rceil+y\omega_D\right) \geq m $.
\item Let $ \alpha = (\lceil y\xi_D \rceil+k)+y\omega_D $ with $ y, k \in \Z_{\geq 0} $. If $ \alpha/\alpha' \leq \varepsilon_+ $, then
\[
	y < \frac{(k+1)\varepsilon_+-k}{\xi_D+\omega_D}.
\]
\item Let $ \alpha_1 = (\lceil y_1\xi_D \rceil+k_1)+y_1\omega_D $ and $ \alpha_2 = (\lceil y_2\xi_D \rceil+k_2)+y_2\omega_D $ with $ k_1, y_1 \in \Z_{\geq 0} $ and $ k_2, y_2 \in \Z_{\geq 0} $. We have
\[
	y_1 \leq y_2\text{ and }k_1 < k_2 \Longrightarrow \alpha_1 \prec \alpha_2.
\]
\item If $ \alpha_1 $ and $ \alpha_2 $ are as in (iii), then
\[
	\alpha_1 \preceq \alpha_2 \Longrightarrow k_1 \leq k_2.
\]
\end{enumerate}
\end{lemma}
\begin{proof}
We prove (i) by induction. If $ m = 1 $, then we can choose $ y = 0 $. Suppose that the statement holds for $ m $ and let $ y_0 \in \Z_{\geq 0} $ be such that $ \alpha_0 = \lceil y_0\xi_D \rceil+y_0\omega_D $ satisfies $ p_K(\alpha_0) \geq m $. We will find $ y \in \Z_{\geq 1} $ such that if we let $ \alpha = \lceil y\xi_D \rceil+y\omega_D $, then $ \alpha \succ \alpha_0 $. It follows that $ p_K(\alpha) > p_K(\alpha_0) $, hence $ p_K(\alpha) \geq m+1 $.

Let $ \delta = y_0\xi_D-\lfloor y_0\xi_D \rfloor $. We know from the theory of rational approximation that there exists $ y \in \Z_{\geq 1} $, $ y > y_0 $ such that $ y\xi_D-\lfloor y\xi_D \rfloor < \delta $.

Since $ y > y_0 $, we get
\[
	\alpha = \lceil y\xi_D \rceil+y\omega_D > \lceil y_0\xi_D \rceil+y_0\omega_D = \alpha_0.
\]
The condition $ \alpha' > \alpha_0' $ is equivalent to $ \lceil y\xi_D \rceil-y\xi_D > \lceil y_0\xi_D \rceil-y_0\xi_D $, which is in turn equivalent to
\[
	y\xi_D-\lfloor y\xi_D \rfloor < y_0\xi_D-\lfloor y_0\xi_D \rfloor.
\]
This inequality is also satisfied due to our choice of $ y $, which finishes the proof of (i).

Next, we prove (ii). If $ \alpha/\alpha' \leq \varepsilon_+ $, then
\[
	\varepsilon_+ \geq \frac{\alpha}{\alpha'} = \frac{(\lceil y\xi_D \rceil+k)+y\omega_D}{(\lceil y\xi_D \rceil+k)+y\omega_D'} > \frac{y(\xi_D+\omega_D)+k}{(y\xi_D+1+k)+y\omega_D'} = \frac{y(\xi_D+\omega_D)+k}{k+1}
\]
since $ \xi_D = -\omega_D' $ by definition. The statement follows.

To prove (iii), we show that $ \alpha_2-\alpha_1 \succ 0 $. First,
\[
	\alpha_2-\alpha_1 = (\lceil y_2\xi_D \rceil-\lceil y_1\xi_D \rceil+k_2-k_1)+(y_2-y_1)\omega_D \geq (k_2-k_1)+(y_2-y_1)\omega_D > 0,
\]
and secondly,
\begin{align*}
	(\alpha_2-\alpha_1)'& = (\lceil y_2\xi_D \rceil-\lceil y_1\xi_D \rceil+k_2-k_1)+(y_2-y_1)\omega_D'\\
	& > (y_2\xi_D-y_1\xi_D-1+k_2-k_1)+(y_2-y_1)\omega_D' = k_2-k_1-1 \geq 0.
\end{align*}

Finally, we prove (iv). If $ \alpha_1 \preceq \alpha_2 $, then
\[
	k_1 \leq \left(\lceil y_1\xi_D \rceil+k_1\right)-y_1\xi_D = \alpha_1' \leq \alpha_2' = \left(\lceil y_2\xi_D \rceil+k_2\right)-y_2\xi_D < k_2+1,
\]
hence $ k_1 \leq k_2 $.
\end{proof}

To construct a table of all elements $ \alpha \in \O_K^+ $ (up to conjugation and multiplication by units) with $ p_K(\alpha) \leq m $, one can proceed as follows:
\begin{itemize}
\item Find $ k_{\max} $ such that $ p_K\left(k_{\max}\right)\geq m $.
\item Find $ y_{\max} $ such that $ y_{\max} \geq \left\lfloor\frac{\varepsilon_+}{\xi_D+\omega_D}\right\rfloor $ and $ p_K\left(\lceil y_{\max}\xi_D\rceil+y_{\max}\omega_D\right) \geq m $.
\item Compute the values of $ p_K(\alpha) $ for $ \alpha = \left(\lceil y\xi_D\rceil+k\right)+y\omega_D $ with $ k \leq k_{\max} $ and $ y \leq y_{\max} $.
\end{itemize}

The correctness of this procedure is justified by the following lemma.

\begin{lemma}
\label{lemmaAux}
Let $ K = \Q(\sqrt{D}) $ where $ D \in \Z_{\geq 2} $ is squarefree. Let $ \alpha = \left(\lceil y\xi_D\rceil+k\right)+y\omega_D $ where $ y, k \in \Z_{\geq 0} $. If $ \alpha/\alpha' \leq \varepsilon_+ $ and $ p_K(\alpha) \leq m $, then $ k \leq k_{\max} $ and $ y \leq y_{\max} $.
\end{lemma}
\begin{proof}
We show that if $ (k, y) $ is outside of the specified range, i.e., $ k > k_{\max} $ or $ y > y_{\max} $, then either $ \alpha/\alpha' > \varepsilon_+ $ or $ p_K(\alpha) > m $.

If $ k = 0 $ and $ y > y_{\max} $, then
\[
	y \geq y_{\max}+1 \geq \left\lfloor\frac{\varepsilon_+}{\xi_D+\omega_D}\right\rfloor+1 > \frac{\varepsilon_+}{\xi_D+\omega_D},
\]
hence $ \alpha/\alpha' > \varepsilon_+ $ by Lemma~\ref{lemmaK} (ii).

If $ k \geq 1 $ and $ y > y_{\max} $, then $ \alpha \succ \lceil y_{\max}\xi_D\rceil+y_{\max}\omega_D $ by Lemma~\ref{lemmaK} (iii), and hence
\[
	p_K(\alpha) > p_K\left(\lceil y_{\max}\xi_D\rceil+y_{\max}\omega_D\right) \geq m.
\]

If $ k > k_{\max} $, then $ \alpha \succ k_{\max} $, hence
\[
	p_K(\alpha) > p_K\left(k_{\max}\right) \geq m.
\]
This completes the proof.
\end{proof}

We performed the computation for $ m = 11 $ and the values of $ D $ not covered by Corollary~\ref{corP}, i.e.,
\[
	D \in \{2, 3, 5, 6, 7, 13, 17, 21\}.
\]
We chose $ k_{\max} = 6 $ since $ p_K(6) \geq p(6) = 11 $, and the values of $ y_{\max} $ were chosen according to Table~\ref{tabYmax}. The results are contained in Tables~\ref{tabD2}--\ref{tabD21} in the Appendix. Based on these tables, we identified all $ \alpha \in \O_K^+ $ with a given number of partitions in each of the quadratic fields, see Tables~\ref{tabD23mod4} and~\ref{tabD1mod4}.

To prove Theorem~\ref{thmD}, one simply examines these tables. For example, let us determine $ \D(m) $ for $ m = 3 $. By Corollary~\ref{corP}, $ 3 \in p_K(\O_K^+) $ for $ D \in \{2, 3\} $ and $ D>5 $. Table~\ref{tabD1mod4} shows that there is no element with $3$ partitions for $ D=5 $. Thus $ \D(3) = \{5\} $.

\begin{table}[ht]
\centering
\begin{TAB}{|l|r|r|r|r|}{|c|cccccccc|}
$D$&$\varepsilon_+$&$\left\lfloor\frac{\varepsilon_+}{\xi_D+\omega_D}\right\rfloor$&$y_{\max}$&$p_K\left(\lceil y_{\max}\xi_D \rceil+y_{\max}\omega_D\right)$\\
$2$&$3+2\sqrt{2}$&$2$&$15$&$16$\\
$3$&$2+\sqrt{3}$&$1$&$11$&$16$\\
$5$&$\frac{3+\sqrt{5}}{2}$&$1$&$13$&$18$\\
$6$&$5+2\sqrt{6}$&$2$&$9$&$12$\\
$7$&$8+3\sqrt{7}$&$3$&$11$&$16$\\
$13$&$\frac{11+3\sqrt{13}}{2}$&$3$&$10$&$12$\\
$17$&$33+8\sqrt{17}$&$16$&$18$&$23$\\
$21$&$\frac{5+\sqrt{21}}{2}$&$1$&$9$&$12$\\
\end{TAB}
\caption{The smallest totally positive unit $ \varepsilon_+ $ and $ y_{\max} $ for different values of $ D $.}
\label{tabYmax}
\end{table}

\section{Totally positive integers with four and six partitions}
\label{secVal2}

Next we treat the first two values not covered by the discussion is Section~\ref{secVal1}, namely $ m = 4 $ and $ 6 $.

\begin{theorem}
\label{thmM4}
Let $ K = \Q(\sqrt{D}) $ where $ D \in \Z_{\geq 2} $ is squarefree. If $ \alpha = (\lceil\xi_D\rceil+2)+\omega_D $, then $ p_K(\alpha) = 4 $.
\end{theorem}
\begin{proof}
First, $ p_K(\alpha) \geq 4 $ because $ \alpha $ has the four partitions
\[
	\lceil\xi_D\rceil+2+\omega_D = \left(\lceil\xi_D\rceil+1+\omega_D\right)+1 = (\lceil\xi_D\rceil+\omega_D)+2 =  (\lceil\xi_D\rceil+\omega_D)+1+1.
\]

It remains to show $ p_K(\alpha) \leq 4 $. Consider an arbitrary partition
\[
	\alpha = \alpha_1+\alpha_2+\dots+\alpha_j,\qquad \alpha_i \in \O_K^+,
\]
and let $ \alpha_i = x_i+y_i\omega_D $. Since $ \alpha $ expressed in the basis $ (1, \omega_D) $ has the second coefficient equal to $ 1 $, at least one of the coefficients $ y_i $ must be positive. Suppose that $ y_1 \geq 1 $, so that $ \alpha_1 $ can be written as $ \alpha_1 = \left(\lceil y_1\xi_D\rceil+k_1\right)+y_1\omega_D $ for some $ k_1 \in \Z_{\geq 0} $. By Lemma~\ref{lemmaK} (iv), the only options for $ k_1 $ are $ k_1 \in \{0, 1, 2\} $.

If $ y_1 \geq 2 $, then $ \alpha_1 = \left(\lceil y_1\xi_D \rceil+k_1\right)+y_1\omega_D \geq \lceil 2\xi_D \rceil+2\omega_D $, and
\[
	\alpha-\left(\lceil 2\xi_D \rceil+2\omega_D\right) = \left(\lceil\xi_D\rceil-\lceil 2\xi_D \rceil+2\right)-\omega_D.
\]

For $ D \notin\{2, 3, 5\} $, this can be estimated as
\[
	\left(\lceil\xi_D\rceil-\lceil 2\xi_D \rceil+2\right)-\omega_D \leq 2-\omega_D < 0.
\]

For $ D \in \{2, 3, 5\} $, we get
\begin{align*}
	\left(\lceil \xi_2 \rceil-\lceil 2\xi_2 \rceil+2\right)-\omega_2 = \left(\left\lceil \sqrt{2}\right\rceil-\left\lceil 2\sqrt{2} \right\rceil+2\right)-\sqrt{2} = 1-\sqrt{2} < 0,\\
	\left(\lceil \xi_3 \rceil-\lceil 2\xi_3 \rceil+2\right)-\omega_3 = \left(\left\lceil \sqrt{3}\right\rceil-\left\lceil 2\sqrt{3} \right\rceil+2\right)-\sqrt{3} = -\sqrt{3} < 0,\\
	\left(\lceil \xi_5 \rceil-\lceil 2\xi_5 \rceil+2\right)-\omega_5 = \left(\left\lceil \frac{\sqrt{5}-1}{2}\right\rceil-\left\lceil \sqrt{5}-1 \right\rceil+2\right)-\frac{1+\sqrt{5}}{2} = \frac{1-\sqrt{5}}{2} < 0.
\end{align*}

We see that in either case, $ \alpha < \lceil 2\xi_D \rceil+2\omega_D $, hence also $ \alpha < \alpha_1 $.

The only remaining possibilities are $ y_1 = 1 $ and $ k_1 \in \{0, 1, 2\} $.
\begin{itemize}
\item If $ k_1 = 0 $, then $ \alpha_1 = \lceil \xi_D \rceil+\omega_D $, hence $ \alpha-\alpha_1 = 2 $. This leads to two partitions of $ \alpha $, namely $ \left( \lceil \xi_D \rceil+\omega_D\right)+2 $ and $ \left( \lceil \xi_D \rceil+\omega_D\right)+1+1 $.
\item If $ k_1 = 1 $, then $ \alpha_1 = \left(\lceil \xi_D \rceil+1\right)+\omega_D $, leading to the partition $ \left( \lceil \xi_D \rceil+1+\omega_D\right)+1 $.
\item If $ k_1 = 2 $, then $ \alpha_1 = \alpha $, and we get the trivial partition.
\end{itemize}
\end{proof}

\begin{theorem}
\label{thmM6}
Let $ K = \Q(\sqrt{D}) $ where $ D \in \Z_{\geq 2} $ is squarefree, $ D \neq 5 $ and let $ \alpha = \left(\lceil 2\xi_D \rceil+2\right)+2\omega_D $.
\begin{enumerate}[i)]
\item If $ \lceil\xi_D\rceil-\xi_D > \frac{1}{2} $, then $ p_K(\alpha)=6 $.
\item If $ \lceil\xi_D\rceil-\xi_D < \frac{1}{2} $, then $ p_K(\alpha)=9 $.
\end{enumerate}
\end{theorem}
\begin{proof}
We note that
\[
	\lceil 2\xi_D \rceil = \begin{cases}
		2\lceil \xi_D \rceil-1,& \text{if }\lceil\xi_D\rceil-\xi_D > \frac{1}{2},\\
		2\lceil \xi_D \rceil,& \text{if }\lceil\xi_D\rceil-\xi_D < \frac{1}{2}.
	\end{cases}
\]

If $ \lceil\xi_D\rceil-\xi_D > \frac{1}{2} $, then $ \alpha = (\lceil 2\xi_D\rceil+2)+2\omega_D = 2\lceil \xi_D\rceil+1+2\omega_D $, and the $ 6 $ partitions of $ \alpha $ are
\begin{align*}
	\lceil 2\xi_D\rceil+2+2\omega_D& = \left(\lceil 2\xi_D\rceil+1+2\omega_D\right)+1 = \left(\lceil 2\xi_D\rceil+2\omega_D\right)+2\\
	& = \left(\lceil 2\xi_D\rceil+2\omega_D\right)+1+1 = \left(\lceil\xi_D\rceil+1+\omega_D\right)+\left(\lceil\xi_D\rceil+\omega_D\right)\\
	& = \left(\lceil\xi_D\rceil+\omega_D\right)+\left(\lceil\xi_D\rceil+\omega_D\right)+1.
\end{align*}

If $ \lceil\xi_D\rceil-\xi_D < \frac{1}{2} $, then $ \alpha = \left(\lceil 2\xi_D\rceil+2\right)+2\omega_D = 2\left(\lceil \xi_D \rceil+1+\omega_D\right) $, and the $ 9 $ partitions of $ \alpha $ are
\begin{align*}
	\lceil 2\xi_D\rceil+2+2\omega_D& = \left(\lceil 2\xi_D\rceil+1+2\omega_D\right)+1 = \left(\lceil 2\xi_D\rceil+2\omega_D\right)+2\\
	& = \left(\lceil 2\xi_D\rceil+2\omega_D\right)+1+1 = \left(\lceil\xi_D\rceil+2+\omega_D\right)+\left(\lceil\xi_D\rceil+\omega_D\right)\\
	& = \left(\lceil\xi_D\rceil+1+\omega_D\right)+\left(\lceil\xi_D\rceil+1+\omega_D\right) = \left(\lceil\xi_D\rceil+1+\omega_D\right)+\left(\lceil\xi_D\rceil+\omega_D\right)+1\\
	& = \left(\lceil\xi_D\rceil+\omega_D\right)+\left(\lceil\xi_D\rceil+\omega_D\right)+2 = \left(\lceil\xi_D\rceil+\omega_D\right)+\left(\lceil\xi_D\rceil+\omega_D\right)+1+1.
\end{align*}

To show that there do not exist any other partitions besides those listed, let
\[
	\alpha = \alpha_1+\alpha_2+\dots+\alpha_j,\qquad \alpha_i \in \O_K^+,
\]
and let $ \alpha_i = x_i+y_i\omega_D $. In the expression of $ \alpha $ in the integral basis $ (1, \omega_D) $, the second coefficient is equal to $ 2 $, which means that at least one of the coefficients $ y_i $ must be positive. We assume $ y_1 \geq 1 $ and write $ \alpha_1 $ as $ \alpha_1 = \left(\lceil y_1\xi_D \rceil+k_1\right)+y_1\omega_D $ for some $ k_1 \in \Z_{\geq 0} $. By Lemma~\ref{lemmaK} (iv), the only options for $ k_1 $ are $ k_1 \in \{0, 1, 2\} $.

If $ y_1 \geq 3 $, then $ \alpha_1 = \left(\lceil y_1\xi_D \rceil+k_1\right)+y_1\omega_D \geq \lceil 3\xi_D \rceil+3\omega_D $, and
\[
	\alpha-\left(\lceil 3\xi_D \rceil+3\omega_D\right) = \left(\lceil 2\xi_D \rceil-\lceil 3\xi_D \rceil+2\right)-\omega_D.
\]
For $ D \notin \{2, 3, 5\} $, this can be estimated as
\[
	 \left(\lceil 2\xi_D \rceil-\lceil 3\xi_D \rceil+2\right)-\omega_D \leq 2-\omega_D < 0.
\]
For $ D \in \{2, 3\} $ (the value $ D = 5 $ being excluded by assumption), we get
\begin{align*}
	\left(\lceil 2\xi_2 \rceil-\lceil 3\xi_2 \rceil+2\right)-\omega_2& = \left(\left\lceil 2\sqrt{2} \right\rceil-\left\lceil 3\sqrt{2} \right\rceil+2\right)-\sqrt{2} = -\sqrt{2} < 0,\\
	\left(\lceil 2\xi_3 \rceil-\lceil 3\xi_3 \rceil+2\right)-\omega_3& = \left(\left\lceil 2\sqrt{3} \right\rceil-\left\lceil 3\sqrt{3} \right\rceil+2\right)-\sqrt{3} = -\sqrt{3} < 0.
\end{align*}
In either case, $ \alpha < \lceil 3\xi_D\rceil+3\omega_D $, hence also $ \alpha < \alpha_1 $. This leaves the values $ y_1 \in \{1, 2\} $.

If $ k_1 = 2 $ and $ y_1 = 1 $, so that $ \alpha_1 = \left(\lceil\xi_D\rceil+2\right)+\omega_D $, we distinguish two cases. If $ \lceil \xi_D \rceil-\xi_D > \frac{1}{2} $, then
\[
	\alpha-\alpha_1 = \left(2\lceil\xi_D\rceil+1+2\omega_D\right)-\left(\lceil\xi_D\rceil+2+\omega_D\right) = \left(\lceil\xi_D\rceil-1\right)+\omega_D,
\]
which is not totally positive. If $ \lceil \xi_D \rceil-\xi_D < \frac{1}{2} $, then
\[
	\alpha-\alpha_1 = \left(2\lceil\xi_D\rceil+2+2\omega_D\right)-\left(\lceil\xi_D\rceil+2+\omega_D\right) = \lceil\xi_D\rceil+\omega_D,
\]
and we obtain the partition $ \alpha = \left(\lceil\xi_D\rceil+2+\omega_D\right)+\left(\lceil\xi_D\rceil+\omega_D\right) $.

The remaining possibilities are $ (k_1, y_1) \in \{(0, 1), (1, 1), (0, 2), (1, 2), (2, 2)\} $, leading to the partitions above.
\end{proof}

Theorem~\ref{thm46} now follows from Theorems~\ref{thmM4} and~\ref{thmM6}.

\begin{remark}
If $ D = 5 $, then $ \alpha = \left(\lceil 2\xi_5 \rceil+2\right)+2\omega_5 = 4+2\cdot\frac{1+\sqrt{5}}{2} $. According to Table~\ref{tabD5}, we have $ p_K(\alpha) = 10 $. The additional partition not listed in the preceding proof is
\[
	\alpha = \left(2+3\cdot\frac{1+\sqrt{5}}{2}\right)+\left(2-\frac{1+\sqrt{5}}{2}\right).
\]
\end{remark}

We mentioned in the introduction that the condition $ \lceil\xi_D\rceil-\xi_D > \frac{1}{2} $ is sufficient but not necessary for $ 6 $ to be contained in the range of $ p_K $, as showed by $ p_K(4) = 6 $ in $ K = \Q(\sqrt{3}) $. The next counterexample is the following one.

\begin{example}
\label{example6}
Let $ D = 14 $. We have $ \lceil\xi_D\rceil-\xi_D = \lceil\sqrt{14}\rceil-\sqrt{14} \leq \frac{1}{2} $ and the element $ \alpha = 16+4\sqrt{14} = \left(\lceil 4\sqrt{14}\rceil+1\right)+4\sqrt{14} $ has $ 6 $ partitions:
\begin{align*}
	16+4\sqrt{14}& = \left(15+4\sqrt{14}\right)+1 = \left(12+3\sqrt{14}\right)+\left(4+\sqrt{14}\right) = \left(8+2\sqrt{14}\right)+\left(8+2\sqrt{14}\right)\\
	& = \left(8+2\sqrt{14}\right)+\left(4+\sqrt{14}\right)+\left(4+\sqrt{14}\right)\\
	& = \left(4+\sqrt{14}\right)+\left(4+\sqrt{14}\right)+\left(4+\sqrt{14}\right)+\left(4+\sqrt{14}\right).
\end{align*}
\end{example}

\addresseshere

\appendix

\newpage
\section{Tables}

\begin{table}[ht]
	\begin{tabular}{|l||r|r|r|r|r|r|r|}
	\hline
	&\multicolumn{7}{c|}{$k$}\\
	\hline
	$y$&$0$&$1$&$2$&$3$&$4$&$5$&$6$\\
	\hline\hline
	$0$&$1$&$1$&$2$&$3$&$6$&$10$&$19$\\
	\hline
	$1$&$1$&$2$&$4$&$8$&$16$&$29$&$54$\\
	\hline
	$2$&$1$&$3$&$6$&$12$&$23$&$44$&$81$\\
	\hline
	$3$&$2$&$6$&$13$&$28$&$56$&$107$&$199$\\
	\hline
	$4$&$2$&$6$&$16$&$33$&$69$&$134$&$257$\\
	\hline
	$5$&$4$&$13$&$33$&$73$&$153$&$301$&$577$\\
	\hline
	$6$&$3$&$12$&$33$&$79$&$169$&$346$&$676$\\
	\hline
	$7$&$1$&$8$&$28$&$73$&$172$&$368$&$748$\\
	\hline
	$8$&$6$&$23$&$69$&$169$&$383$&$801$&$1610$\\
	\hline
	$9$&$2$&$16$&$56$&$153$&$368$&$816$&$1692$\\
	\hline
	$10$&$10$&$44$&$134$&$346$&$801$&$1732$&$3544$\\
	\hline
	$11$&$4$&$29$&$107$&$301$&$748$&$1692$&$3595$\\
	\hline
	$12$&$1$&$19$&$81$&$257$&$676$&$1610$&$3544$\\
	\hline
	$13$&$8$&$54$&$199$&$577$&$1458$&$3369$&$7276$\\
	\hline
	$14$&$3$&$34$&$149$&$475$&$1285$&$3109$&$6981$\\
	\hline
	$15$&$16$&$98$&$365$&$1071$&$2760$&$6471$&$14201$\\
	\hline
	\end{tabular}
\caption{Number of partitions of $ \alpha = (\lceil y\sqrt{2}\rceil+k)+y\sqrt{2} $ in $ K = \Q(\sqrt{2}) $.}
\label{tabD2}
\end{table}

\begin{table}[ht]
	\begin{tabular}{|l||r|r|r|r|r|r|r|}
	\hline
	&\multicolumn{7}{c|}{$k$}\\
	\hline
	$y$&$0$&$1$&$2$&$3$&$4$&$5$&$6$\\
	\hline\hline
	$0$&$1$&$1$&$2$&$3$&$6$&$10$&$18$\\
	\hline
	$1$&$1$&$2$&$4$&$7$&$14$&$25$&$45$\\
	\hline
	$2$&$2$&$4$&$9$&$16$&$32$&$57$&$103$\\
	\hline
	$3$&$3$&$7$&$16$&$32$&$64$&$118$&$215$\\
	\hline
	$4$&$1$&$6$&$14$&$32$&$64$&$128$&$237$\\
	\hline
	$5$&$2$&$10$&$25$&$57$&$118$&$237$&$447$\\
	\hline
	$6$&$4$&$18$&$45$&$103$&$215$&$432$&$819$\\
	\hline
	$7$&$7$&$29$&$76$&$177$&$376$&$760$&$1456$\\
	\hline
	$8$&$2$&$14$&$52$&$133$&$309$&$656$&$1328$\\
	\hline
	$9$&$4$&$25$&$87$&$224$&$521$&$1115$&$2262$\\
	\hline
	$10$&$9$&$45$&$149$&$378$&$878$&$1876$&$3811$\\
	\hline
	$11$&$16$&$76$&$244$&$624$&$1448$&$3105$&$6317$\\
	\hline
	\end{tabular}
\caption{Number of partitions of $ \alpha = (\lceil y\sqrt{3}\rceil+k)+y\sqrt{3} $ in $ K = \Q(\sqrt{3}) $.}
\label{tabD3}
\end{table}

\begin{table}[ht]
	\begin{tabular}{|l||r|r|r|r|r|r|r|}
	\hline
	&\multicolumn{7}{c|}{$k$}\\
	\hline
	$y$&$0$&$1$&$2$&$3$&$4$&$5$&$6$\\
	\hline\hline
	$0$&$1$&$1$&$2$&$4$&$8$&$14$&$29$\\
	\hline
	$1$&$1$&$2$&$4$&$9$&$18$&$36$&$71$\\
	\hline
	$2$&$2$&$4$&$10$&$21$&$43$&$84$&$166$\\
	\hline
	$3$&$1$&$4$&$9$&$21$&$46$&$92$&$183$\\
	\hline
	$4$&$2$&$8$&$18$&$43$&$92$&$191$&$377$\\
	\hline
	$5$&$4$&$14$&$36$&$84$&$183$&$377$&$753$\\
	\hline
	$6$&$2$&$9$&$29$&$71$&$166$&$356$&$737$\\
	\hline
	$7$&$4$&$18$&$54$&$136$&$313$&$678$&$1396$\\
	\hline
	$8$&$1$&$10$&$36$&$106$&$259$&$592$&$1269$\\
	\hline
	$9$&$4$&$21$&$71$&$198$&$484$&$1093$&$2341$\\
	\hline
	$10$&$9$&$43$&$136$&$371$&$890$&$2003$&$4257$\\
	\hline
	$11$&$2$&$21$&$84$&$259$&$683$&$1623$&$3613$\\
	\hline
	$12$&$8$&$46$&$166$&$484$&$1250$&$2926$&$6467$\\
	\hline
	$13$&$18$&$92$&$313$&$890$&$2246$&$5217$&$11429$\\
	\hline
	\end{tabular}
\caption{Number of partitions of $ \alpha = (\lceil y\omega_5\rceil+k)+y\omega_5 $ in $ K = \Q(\sqrt{5}) $.}
\label{tabD5}
\end{table}

\begin{table}[ht]
	\begin{tabular}{|l||r|r|r|r|r|r|r|}
	\hline
	&\multicolumn{7}{c|}{$k$}\\
	\hline
	$y$&$0$&$1$&$2$&$3$&$4$&$5$&$6$\\
	\hline\hline
	$0$&$1$&$1$&$2$&$3$&$5$&$7$&$12$\\
	\hline
	$1$&$1$&$2$&$4$&$7$&$12$&$20$&$34$\\
	\hline
	$2$&$1$&$3$&$6$&$12$&$21$&$36$&$60$\\
	\hline
	$3$&$2$&$6$&$13$&$26$&$48$&$85$&$146$\\
	\hline
	$4$&$2$&$6$&$16$&$33$&$65$&$117$&$208$\\
	\hline
	$5$&$4$&$13$&$33$&$70$&$138$&$255$&$456$\\
	\hline
	$6$&$3$&$12$&$33$&$78$&$160$&$309$&$567$\\
	\hline
	$7$&$7$&$26$&$70$&$161$&$332$&$642$&$1184$\\
	\hline
	$8$&$5$&$21$&$65$&$160$&$353$&$708$&$1355$\\
	\hline
	$9$&$12$&$48$&$138$&$332$&$719$&$1438$&$2738$\\
	\hline
	\end{tabular}
\caption{Number of partitions of $ \alpha = (\lceil y\sqrt{6}\rceil+k)+y\sqrt{6} $ in $ K = \Q(\sqrt{6}) $.}
\label{tabD6}
\end{table}

\begin{table}[ht]
	\begin{tabular}{|l||r|r|r|r|r|r|r|}
	\hline
	&\multicolumn{7}{c|}{$k$}\\
	\hline
	$y$&$0$&$1$&$2$&$3$&$4$&$5$&$6$\\
	\hline\hline
	$0$&$1$&$1$&$2$&$3$&$5$&$7$&$12$\\
	\hline
	$1$&$1$&$2$&$4$&$7$&$12$&$19$&$32$\\
	\hline
	$2$&$2$&$4$&$9$&$16$&$29$&$48$&$82$\\
	\hline
	$3$&$1$&$4$&$9$&$19$&$36$&$64$&$110$\\
	\hline
	$4$&$2$&$8$&$18$&$39$&$74$&$135$&$234$\\
	\hline
	$5$&$4$&$14$&$34$&$73$&$143$&$264$&$468$\\
	\hline
	$6$&$2$&$9$&$29$&$67$&$144$&$279$&$519$\\
	\hline
	$7$&$4$&$18$&$53$&$125$&$266$&$521$&$972$\\
	\hline
	$8$&$9$&$34$&$99$&$229$&$489$&$958$&$1798$\\
	\hline
	$9$&$3$&$19$&$67$&$182$&$420$&$884$&$1738$\\
	\hline
	$10$&$7$&$39$&$125$&$332$&$754$&$1582$&$3101$\\
	\hline
	$11$&$16$&$73$&$229$&$588$&$1332$&$2777$&$5452$\\
	\hline
	\end{tabular}
\caption{Number of partitions of $ \alpha = (\lceil y\sqrt{7}\rceil+k)+y\sqrt{7} $ in $ K = \Q(\sqrt{7}) $.}
\label{tabD7}
\end{table}

\begin{table}[ht]
	\begin{tabular}{|l||r|r|r|r|r|r|r|}
	\hline
	&\multicolumn{7}{c|}{$k$}\\
	\hline
	$y$&$0$&$1$&$2$&$3$&$4$&$5$&$6$\\
	\hline\hline
	$0$&$1$&$1$&$2$&$3$&$5$&$8$&$14$\\
	\hline
	$1$&$1$&$2$&$4$&$7$&$13$&$23$&$40$\\
	\hline
	$2$&$1$&$3$&$6$&$12$&$22$&$40$&$70$\\
	\hline
	$3$&$1$&$3$&$8$&$16$&$31$&$58$&$105$\\
	\hline
	$4$&$3$&$8$&$20$&$40$&$79$&$146$&$265$\\
	\hline
	$5$&$2$&$8$&$21$&$48$&$98$&$191$&$355$\\
	\hline
	$6$&$2$&$8$&$24$&$56$&$121$&$240$&$460$\\
	\hline
	$7$&$6$&$21$&$58$&$132$&$280$&$554$&$1052$\\
	\hline
	$8$&$4$&$20$&$58$&$145$&$318$&$656$&$1275$\\
	\hline
	$9$&$3$&$16$&$56$&$148$&$345$&$736$&$1485$\\
	\hline
	$10$&$12$&$48$&$145$&$357$&$803$&$1669$&$3306$\\
	\hline
	\end{tabular}
\caption{Number of partitions of $ \alpha = (\lceil y\omega_{13}\rceil+k)+y\omega_{13} $ in $ K = \Q(\sqrt{13}) $.}
\label{tabD13}
\end{table}

\begin{table}[ht]
	\begin{tabular}{|l||r|r|r|r|r|r|r|}
	\hline
	&\multicolumn{7}{c|}{$k$}\\
	\hline
	$y$&$0$&$1$&$2$&$3$&$4$&$5$&$6$\\
	\hline\hline
	$0$&$1$&$1$&$2$&$3$&$5$&$8$&$14$\\
	\hline
	$1$&$1$&$2$&$4$&$7$&$12$&$21$&$36$\\
	\hline
	$2$&$2$&$4$&$9$&$16$&$30$&$52$&$91$\\
	\hline
	$3$&$1$&$4$&$9$&$19$&$36$&$66$&$118$\\
	\hline
	$4$&$2$&$8$&$18$&$39$&$75$&$141$&$252$\\
	\hline
	$5$&$1$&$5$&$16$&$37$&$78$&$153$&$287$\\
	\hline
	$6$&$3$&$11$&$33$&$74$&$157$&$306$&$577$\\
	\hline
	$7$&$1$&$7$&$24$&$65$&$146$&$305$&$598$\\
	\hline
	$8$&$3$&$16$&$49$&$128$&$282$&$587$&$1145$\\
	\hline
	$9$&$8$&$33$&$98$&$244$&$538$&$1107$&$2160$\\
	\hline
	$10$&$3$&$19$&$70$&$193$&$467$&$1012$&$2071$\\
	\hline
	$11$&$8$&$42$&$139$&$371$&$873$&$1879$&$3810$\\
	\hline
	$12$&$2$&$22$&$90$&$277$&$706$&$1629$&$3456$\\
	\hline
	$13$&$8$&$50$&$185$&$533$&$1324$&$2989$&$6286$\\
	\hline
	$14$&$2$&$23$&$112$&$371$&$1019$&$2455$&$5453$\\
	\hline
	$15$&$8$&$58$&$235$&$728$&$1911$&$4504$&$9834$\\
	\hline
	$16$&$1$&$25$&$132$&$482$&$1398$&$3551$&$8178$\\
	\hline
	$17$&$7$&$64$&$287$&$953$&$2641$&$6501$&$14704$\\
	\hline
	$18$&$23$&$152$&$595$&$1850$&$4915$&$11797$&$26201$\\
	\hline
	\end{tabular}
\caption{Number of partitions of $ \alpha = (\lceil y\omega_{17}\rceil+k)+y\omega_{17} $ in $ K = \Q(\sqrt{17}) $.}
\label{tabD17}
\end{table}

\begin{table}[ht]
	\begin{tabular}{|l||r|r|r|r|r|r|r|}
	\hline
	&\multicolumn{7}{c|}{$k$}\\
	\hline
	$y$&$0$&$1$&$2$&$3$&$4$&$5$&$6$\\
	\hline\hline
	$0$&$1$&$1$&$2$&$3$&$5$&$8$&$14$\\
	\hline
	$1$&$1$&$2$&$4$&$7$&$12$&$21$&$36$\\
	\hline
	$2$&$2$&$4$&$9$&$16$&$29$&$50$&$87$\\
	\hline
	$3$&$3$&$7$&$16$&$31$&$57$&$102$&$179$\\
	\hline
	$4$&$5$&$12$&$29$&$57$&$110$&$198$&$353$\\
	\hline
	$5$&$1$&$8$&$21$&$50$&$102$&$198$&$366$\\
	\hline
	$6$&$2$&$14$&$36$&$87$&$179$&$353$&$656$\\
	\hline
	$7$&$4$&$22$&$60$&$144$&$303$&$602$&$1136$\\
	\hline
	$8$&$7$&$36$&$98$&$238$&$504$&$1013$&$1924$\\
	\hline
	$9$&$12$&$56$&$157$&$381$&$822$&$1661$&$3189$\\
	\hline
\end{tabular}
\caption{Number of partitions of $ \alpha = (\lceil y\omega_{21}\rceil+k)+y\omega_{21} $ in $ K = \Q(\sqrt{21}) $.}
\label{tabD21}
\end{table}

\def\arraystretch{1.3}

\begin{table}[ht]
	\begin{tabular}{|l|r|r|r|r|}
	\hline
	\multicolumn{5}{|c|}{$ \alpha\in\O_K^+:\; p_K(\alpha) = m $}\\
	\hline
	&\multicolumn{4}{|c|}{$D$}\\
	\hline
	$m$&$2$&$3$&$6$&$7$\\
	\hline\hline
	$1$&$1$, $2+\sqrt{2}$&$1$&$1$, $3+\sqrt{6}$&$1$, $3+\sqrt{7}$\\
	\hline
	$2$&$2$, $3+\sqrt{2}$&$2$, $3+\sqrt{3}$&$2$, $4+\sqrt{6}$&$2$, $4+\sqrt{7}$, $6+2\sqrt{7}$\\
	\hline
	$3$&$3$, $4+2\sqrt{2}$&$3$&$3$, $6+2\sqrt{6}$&$3$\\
	\hline
	$4$&$4+\sqrt{2}$&$4+\sqrt{3}$&$5+\sqrt{6}$&$5+\sqrt{7}$, $7+2\sqrt{7}$, $9+3\sqrt{7}$\\
	\hline
	$5$&--&--&$4$&$4$\\
	\hline
	$6$&$4$, $5+2\sqrt{2}$, $6+3\sqrt{2}$&$4$&$7+2\sqrt{6}$, $9+3\sqrt{6}$&--\\
	\hline
	$7$&--&$5+\sqrt{3}$&$5$, $6+\sqrt{6}$&$5$, $6+\sqrt{7}$\\
	\hline
	$8$&$5+\sqrt{2}$&--&--&$12+4\sqrt{7}$\\
	\hline
	$9$&--&$6+2\sqrt{3}$&--&$8+2\sqrt{7}$, $10+3\sqrt{7}$\\
	\hline
	$10$&$5$&$5$&--&--\\
	\hline
	$11$&--&--&--&--\\
	\hline
	\end{tabular}
\caption{All elements (up to conjugation and multiplication by units) with $ m $ partitions in $ K = \Q(\sqrt{D}) $ for $ D \equiv 2, 3 \pmod{4} $.}
\label{tabD23mod4}
\end{table}

\begin{table}[ht]
	\begin{tabular}{|l|r|r|r|r|}
	\hline
	\multicolumn{5}{|c|}{$ \alpha\in\O_K^+:\; p_K(\alpha) = m $}\\
	\hline
	&\multicolumn{4}{|c|}{$D$}\\
	\hline
	$m$&$5$&$13$&$17$&$21$\\
	\hline\hline
	$1$&$1$&$1$, $\frac{5+\sqrt{13}}{2}$&$1$, $\frac{5+\sqrt{17}}{2}$, $\frac{13+3\sqrt{17}}{2}$&$1$\\
	\hline
	$2$&$2$, $\frac{5+\sqrt{5}}{2}$&$2$, $\frac{7+\sqrt{13}}{2}$&$2$, $\frac{7+\sqrt{17}}{2}$, $5+\sqrt{17}$, $9+2\sqrt{17}$&$2$, $\frac{7+\sqrt{21}}{2}$\\
	\hline
	$3$&--&$3$, $5+\sqrt{13}$, $\frac{13+3\sqrt{13}}{2}$&$3$, $13+3\sqrt{17}$, $17+4\sqrt{17}$&$3$\\
	\hline
	$4$&$3$, $\frac{7+\sqrt{5}}{2}$&$\frac{9+\sqrt{13}}{2}$&$\frac{9+\sqrt{17}}{2}$, $6+\sqrt{17}$, $\frac{15+3\sqrt{17}}{2}$&$\frac{9+\sqrt{21}}{2}$\\
	\hline
	$5$&--&$4$&$4$, $\frac{23+5\sqrt{17}}{2}$&$4$\\
	\hline
	$6$&--&$6+\sqrt{13}$&--&--\\
	\hline
	$7$&--&$\frac{11+\sqrt{13}}{2}$&$\frac{11+\sqrt{17}}{2}$, $\frac{31+7\sqrt{17}}{2}$&$\frac{11+\sqrt{21}}{2}$\\
	\hline
	$8$&$4$&$5$, $\frac{15+3\sqrt{13}}{2}$, $9+2\sqrt{13}$&$5$, $10+2\sqrt{17}$, $\frac{39+9\sqrt{17}}{2}$, $\frac{47+11\sqrt{17}}{2}$&$5$\\
	\hline
	$9$&$\frac{9+\sqrt{5}}{2}$&--&$7+\sqrt{17}$, $\frac{17+3\sqrt{17}}{2}$&$7+\sqrt{21}$\\
	\hline
	$10$&$5+\sqrt{5}$&--&--&--\\
	\hline
	$11$&--&--&$14+3\sqrt{17}$&--\\
	\hline
	\end{tabular}
\caption{All elements (up to conjugation and multiplication by units) with $ m $ partitions in $ K = \Q(\sqrt{D}) $ for $ D \equiv 1 \pmod{4} $.}
\label{tabD1mod4}
\end{table}

\end{document}